\documentclass[a4paper,12pt]{amsart}
\usepackage{amsmath,amsfonts,amssymb,amsthm}
\usepackage{graphics}
\usepackage{epsfig}
\usepackage{esint}
\usepackage{shuffle}
\usepackage{istgame}
\usepackage{vmargin}
\setpapersize{A4}
\setmargins{2.5cm}{1.5cm}{16.5cm}{23.42cm}{10pt}{1cm}{0pt}{2cm}

\usepackage{enumerate}

\newcommand{ \E }{ \mathbb{E} }

\newcommand{\ignore}[1]{}

\usepackage[hidelinks,hyperindex,breaklinks]{hyperref}

\newtheorem{proposition}{Proposition}
\newtheorem{theorem}{Theorem}

\newtheorem{lemma}{Lemma}
\newtheorem{corollary}{Corollary}

\begin{document}

\title[A critical drift-diffusion equation: connections to the diffusion on $ \textbf{SL}(2)$]{A critical drift-diffusion equation:\\ connections to the diffusion on $ \textbf{SL}(2)$}
\author[P. Morfe]{Peter S.~Morfe}
\author[F. Otto]{Felix Otto}
\author[C. Wagner]{Christian Wagner}

\maketitle

\begin{abstract}
In this note, we connect two seemingly unrelated objects:
On the one hand is a two-dimensional drift-diffusion process $X$ with divergence-free
and time-independent drift $b$. The drift is given by a stationary Gaussian ensemble,
and we focus on the critical case where a small-scale cut-off is necessary for well-posedness
and the large-scale cancellations lead to a borderline super-diffusive behavior.
On the other hand is the natural diffusion $F$ on the Lie group $\textbf{SL}(2)$
of matrices of determinant one.
As a consequence of this connection, the strongly non-Gaussian character of $F$ 
transmits to how $X$ depends on its starting point.
\end{abstract}

\section{Introduction and main results}

This note leverages and refines the observations in \cite{OW24}, which itself
is based on \cite{CMOW} (see also \cite{MOW24}
for a more efficient presentation). We refer to these papers for
more context and references.


\subsection{The drift-diffusion process $X$ for a Gaussian ensemble of divergence-free
drifts $b$, and its expected position $u$}\label{ss:X}
The main result of this note connects two seemingly unrelated objects:
On the one hand, we consider the drift-diffusion process
\begin{align}\label{ao02}
dX=b(X)dt+\sqrt{2}dW
\end{align}
in the Euclidean plane,
where $W$ denotes the standard Brownian motion\footnote{the factor of $\sqrt{2}$ in (\ref{ao02})
avoids a factor $\frac{1}{2}$ in (\ref{ao04})}.
We assume that the time-independent drift $b$, a tangent vector field, is divergence-free:
\begin{align}\label{ao24}
\nabla.b=0,
\end{align}
where the dot denotes the natural pairing between the differential $\nabla$,
which we consider a cotangent vector\footnote{since we think of $\nabla u$ as a cotangent
vector}, and the tangent vector $b$ (in view of the upcoming tensor analysis and
elementary representation theory, we are
deliberately careful in distinguishing tangent and cotangent vectors).
We further assume that $b$ is random and of white-noise character modulo the constraint 
(\ref{ao24}). As we shall discuss below, a small-scale cut-off is required for well-posedness.

\medskip

More specifically, the centered and stationary Gaussian ensemble
is characterized by its covariance tensor $c(x-y)$ $=\mathbb{E}b(x)\otimes b(y)$,
an element of the tensor product $(\mbox{tangent space})$ $\otimes(\mbox{tangent space})$,
and a function of the tangent vector $z=x-y$.
It is convenient to specify $c$ on the level of its positive semidefinite Fourier transform
${\mathcal F}c(k)=(2\pi)^{-1}\int dz\exp(-\iota k.z)c(z)$:
%
\begin{align}\label{ao08}
{\mathcal F}c(k)=\varepsilon^2I(|k|\le 1)(\sum_{i=1}^2e_{i}\otimes e_{i}-\frac{k^*\otimes k^*}{|k|^2}),
\end{align}
where $\{e_1,e_2\}$ denotes the Cartesian basis of tangent space and 
$k^*$ is the tangent vector related to the cotangent wave vector $k$ 
by the Euclidean inner product. Since the Fourier transform of the covariance
function of $\nabla.b$ is given by the contraction of $k\otimes k$ with (\ref{ao08}),
and thus vanishes, we learn that (\ref{ao24}) is indeed satisfied.
Since we non-dimensionalized length by setting the small-scale cut-off to unity,
the amplitude $\varepsilon$ is a genuine non-dimensional parameter
(also known as the P\'eclet number);
we will be interested in the regime $0\le\varepsilon\ll 1$.

\medskip

We will focus on the expected position $u(x,t)$ 
(with respect to the thermal noise $W$) of
the solution $X$ of (\ref{ao02}) with $X(t=0)=x$. It is characterized by the
drift-diffusion equation
\begin{align}\label{ao04}
\partial_tu-b.\nabla u-\triangle u=0,\quad u(t=0,x)=x,
\end{align}
where we recall that the dot denotes the natural pairing between the tangent vector $b$
and the cotangent vector $\nabla$.
Let us briefly discuss the need for the small-scale cut-off in (\ref{ao08}):
If it were not present, $b$ would share the scale invariance of spatial white noise;
hence in view of the transport term $b.\nabla u$,
the solution $u$ can be H\"older continuous of exponent $\alpha$
only provided $\alpha<2-\frac{d}{2}=1$ (where $2$ is the order of the equation
and $-\frac{d}{2}$ comes from the scale invariance of spatial white noise).
In fact, we establish that $u$ is Lipschitz continuous up to a logarithm,
see the discussion after Corollary \ref{Cor:1}.
However, this regularity is not sufficient to give a sense to
$b.\nabla u$, even when rewriting it as $\nabla. ub$,
since the sum of the orders of regularity of $u$ and $b$ is (borderline) negative.

\medskip

In view of the divergence condition (\ref{ao24}), a connection between
(\ref{ao02}) and a process on $\textbf{SL}(2)$ is not surprising per se. If for
a moment, we disregard diffusion in (\ref{ao02}) and (\ref{ao04}), we obtain
\begin{align}\label{ao30}
d \nabla u ( 0, t ) = \nabla u ( 0 , t ) \nabla b ( X_t ) dt , 
\quad { \rm where } \quad X_{ t } ~ { \rm starts ~ at ~ the ~ origin } . 
\end{align}
A remark on our tensorial notation is in order: Like for $b$, we think of
increments of $u$ as tangent vectors; since we consider the differential $\nabla$
as a cotangent vector, $\nabla b$ and $\nabla u$ are elements
of the tensor space $(\mbox{tangent space})$ $\otimes(\mbox{cotangent space})$,
which is canonically isomorphic to the endomorphism space ${\rm End}(\mbox{cotangent space})$.
Hence the product in (\ref{ao30}) is a product of endomorphisms.
Our definition of $\nabla u$ is the transpose of the usual definition
of the derivative of a map as an element of ${\rm End}(\mbox{tangent space})$;
we opted in favor of it because the multiplication order in (\ref{ao30}) is more standard
on the stochastic equation level, see (\ref{ao01bis}) and (\ref{ao01}).

\medskip

According to Jacobi's formula, (\ref{ao30}) implies
\begin{align}\label{ao31}
d{\rm det}\nabla u=({\rm det}\nabla u)({\rm tr}\nabla b)dt\quad\mbox{along}\quad x=X_{t}.
\end{align}
Because of ${\rm tr}\nabla b$ $=\nabla.b$, we learn from (\ref{ao24}) 
that $\nabla b(X_t)$ is an element of $\mathfrak{sl}(2)$, the linear space
of trace-free endomorphisms. We then learn from (\ref{ao31})
that ${\rm det}\nabla u=1$ is preserved, i.~e.~$\nabla u(X_t,t)$ is an element of 
$\textbf{SL}(2)$, the space of endomorphisms of determinant one.


\subsection{A natural notion of diffusion $F$ on the Lie group $\textbf{SL}(2)$}\label{ss:F}

On the other hand, we consider a diffusion on the Lie group $\textbf{SL}(2)$, 
as given by the Stratonovich evolution equation
\begin{align}\label{ao01bis}
dF=F\circ dB,
\end{align}
where $B$ is a Brownian motion on the corresponding Lie algebra $\mathfrak{sl}(2)$,
which is a three-dimensional linear space. The tensorial
structure of the r.~h.~s.~of (\ref{ao01bis}) is to be
understood as a product of endomorphisms (of cotangent space). 
The Stratonovich form can be best understood as follows: 
Approximate a realization of $B$ by a smooth path $B'$ in the local uniform topology,
solve (\ref{ao01bis}) classically for some initial data; then almost surely, this solution
$F'$ is locally uniformly close to the Stratonovich solution $F$. In particular,
the Stratonovich form of (\ref{ao01bis}) is compatible with the chain rule,
so that as in (\ref{ao31}) we have $d{\rm det}F$ $=({\rm det}F)\circ d{\rm tr} B$ $=0$. 
Hence (\ref{ao01bis})
indeed defines a time-stationary Markov process on $\textbf{SL}(2)$. Due to the left-invariance of the evolution \eqref{ao01bis}, there is no loss in generality assuming
\begin{align*}
F_{ \tau = 0 } = { \rm id } .
\end{align*}

\medskip

Fixing the Brownian motion $B$ on the linear space $\mathfrak{sl}(2)$
amounts to the choice of the covariance 
\begin{align}\label{ao55}
C =\mathbb{E}B_{\tau=1}\otimes B_{\tau=1}\in \mathfrak{sl}(2)\otimes\mathfrak{sl}(2).
\end{align}
%
Since $ C $ is symmetric and positive semidefinite, it admits a decomposition into symmetric rank-one tensors:
\begin{align}\label{ao39}
C =\int\mu(dE) E\otimes E ~ 
\mbox{for a compactly supported probability measure $\mu$ on $\mathfrak{sl}(2)\ni E$}.
\end{align}
This more general representation will be useful when identifying covariances on the Fourier level
in Subsection \ref{ss:coupl}. 

\medskip

We now argue that there is a canonical choice for this covariance (\ref{ao55}).
We note that $\textbf{SO}(2)$ acts on $\mathfrak{sl}(2)\subset{\rm End}(\mbox{cotangent space})$, 
and thus also on $\mathfrak{sl}(2)\otimes\mathfrak{sl}(2)$, by conjugation.
We shall postulate that $B$ or, equivalently, that the covariance (\ref{ao39}) 
is invariant under the action.
We now appeal to standard representation theory, see for instance
\cite[Appendix A1]{OW24}:
The irreducible decomposition under this action is given by
\begin{align}\label{ao33}
\mathfrak{sl}(2)&=\mbox{space of symmetric trace-free endomorphisms}\nonumber\\
&\oplus\mbox{space of skew symmetric endomorphisms};
\end{align}
a compatible basis (also in the sense that $E_2$ arises from applying
an infinitesimal generator of the group action of $\mathbf{SO}(2)$ to $E_1$) of
$\mathfrak{sl}(2)$ is given by
\begin{align}\label{ao42}
E_1=\left(\begin{array}{rr}1&0\\0&-1\end{array}\right),\quad
E_2=\left(\begin{array}{rr}0&1\\1&0\end{array}\right),\quad
E_3=\left(\begin{array}{rr}0&-1\\1&0\end{array}\right);
\end{align}
and (\ref{ao55}) is of the form (\ref{ao39}) with
\begin{align*}
C=\kappa_{sym}(E_1\otimes E_1+E_2\otimes E_2)+\kappa_{skew}E_3\otimes E_3.
\end{align*}
This clearly is of the form (\ref{ao39}) with
\begin{align}\label{ao41}
\mu=\kappa_{sym}(\delta_{E_1}+\delta_{E_2})+\kappa_{skew}\delta_{E_3}.
\end{align}
Hence there are only two parameters to fix, namely $\kappa_{sym}$ and $\kappa_{skew}$, 
and we now argue how this is done canonically.

\medskip

Our first requirement may be paraphrased by saying that in \eqref{ao01bis} Stratonovich = Itô; more formally we postulate
\begin{align}\label{equiv01}
	d F = F d B
	\quad { \rm and ~ show ~ its ~ equivalence ~ to }  \quad 
	2\kappa_{sym}-\kappa_{skew}=0.
\end{align}
In particular the assumption \eqref{equiv01} implies that $F$ is a martingale in the ambient linear space
${\rm End}(\mbox{cotangent space})$.
Our second requirement is the normalization 
\begin{align}\label{equiv02}
\mathbb{E}|F^{\dagger}_{\tau} x|^2= e^\tau |x|^2
	\quad { \rm which ~ we ~ prove ~ to ~ be ~ equivalent ~ to } \quad 2 \kappa_{sym}+\kappa_{skew}=1 
\end{align}
under the assumption \eqref{equiv01}.
Here $F^\dagger\in{\rm End}(\mbox{tangent space})$ denotes the transpose of 
$F\in{\rm End}(\mbox{cotangent space})$.
Together, the conditions in \eqref{equiv01} and \eqref{equiv02} fix the parameters:
\begin{align}\label{sde01}
	2 \kappa_{ skew } = 4 \kappa_{ sym } = 1 .
\end{align}
We thus have identified a canonical diffusion on $\textbf{SL}(2)$.
The rest of this section is devoted to proving the equivalences in \eqref{equiv01} and \eqref{equiv02}.

\medskip

We now turn to the first equivalence \eqref{equiv01}:  To this end, we need to argue that the It\^{o}-Stratonovich correction vanishes.
For clarity we momentarily consider
the more general form $dF=\sigma(F)\circ dB$, where $\sigma(F)$ now is a linear 
endomorphism of ${\rm End}(\mbox{cotangent}$ $\mbox{space})$. While the It\^{o} formulation
is the limit of a forward Euler discretization $F'=F+\sigma(F)$ $(B'-B)$, the Stratonovich
formulation arises from the (second-order)
mid-point rule $F'=F+\sigma(F$ $+\frac{1}{2}\sigma(F)(B'-B))(B'-B)$.
The leading-order correction, divided by the size of the time interval $\tau'-\tau$,
is given by
$\frac{1}{2}(\sigma(F)( \frac{B'-B}{\sqrt{\tau'-\tau}} )).
\nabla(\sigma(F)(\frac{B'-B}{\sqrt{\tau'-\tau}}))$, 
where $\nabla$ denotes the differential of the map 
$F\mapsto \sigma(F)( \frac{B'-B}{\sqrt{\tau'-\tau}})$
on the linear ${\rm End}(\mbox{cotangent space})$.
By the independence of increments of $B$,
the continuum limit of the correction is given by its expectation
w.~r.~t.~$\frac{B'-B}{\sqrt{\tau'-\tau}}$.
Once more by the independence and stationarity of increments of Brownian motion,
the increment $\frac{B'-B}{\sqrt{\tau'-\tau}}$ has the same law as $\frac{B_{\tau'-\tau}}{\sqrt{\tau'-\tau}}$.
By the scaling invariance of $B$, the covariance of 
$\frac{B_{\tau'-\tau}}{\sqrt{\tau'-\tau}}$ is given by 
$C$ defined in (\ref{ao55}). Hence the expectation of the correction is a contraction of
$C$, which can be conveniently expressed in terms of (\ref{ao39}):
\begin{align*}
\mbox{It\^{o}-Stratonovich correction from $\sigma(F)\circ dB$}= 
\frac{1}{2}\int \mu(dE)\,(\sigma(F)E).\nabla(\sigma(F)E).
\end{align*}
%
For our $\sigma(F)=F$, this turns into
\begin{align}\label{ao49}
\mbox{It\^{o}-Stratonovich correction from $F\circ dB$}= \frac{1}{2}F\int \mu(dE)\,E^2.
\end{align}
It follows from (\ref{ao42}), noting that $E_1$ and $E_2$ are reflections while $E_3$
is the rotation by $\frac{\pi}{2}$, and (\ref{ao41}) that this expression collapses to
\begin{align*}
\mbox{It\^{o}-Stratonovich correction}:= \frac{1}{2} (2\kappa_{sym}-\kappa_{skew})F,
\end{align*}
establishing \eqref{equiv01}.

\medskip

We now turn to the second postulate \eqref{equiv02}. Since the law of $B$, and thus that of $F$, is invariant under conjugation by $\textbf{SO}(2)$,
$|x|^{-2}\mathbb{E}|F^{\dagger}_\tau x|^2$ is independent of $x$. Hence the first item in \eqref{equiv02} is equivalent to
\begin{align}\label{ao35}
\mathbb{E}|F_\tau|^2=2e^\tau,
\end{align}
where $|F|^2 = \sum_{ i = 1 }^{ 2 } | F e^{ i } |^2 $ denotes the squared Frobenius norm, and $ \{ e^1, e^2 \} $ the Cartesian basis of cotangent space.
Hence we need to monitor $\mathbb{E}\zeta(F)$
for $\zeta(F)=|F|^2$, which we do based on the It\^{o} evolution equation \eqref{equiv01}.
Recalling that the It\^{o} formulation in \eqref{equiv01} arises from
the forward Euler scheme $F'=F+F(B'-B)$, we have to second order
$\zeta(F')=\zeta(F)+F(B'-B).\nabla\zeta(F)+\frac{1}{2}\nabla^2\zeta(F)(F(B'-B),F(B'-B))$,
where $\nabla^2\zeta$ denotes the second differential of the map $F\mapsto \zeta(F)$
on the ambient ${\rm End}(\mbox{cotangent space})$. Taking the expectation,
the martingale term $F(B'-B).\nabla\zeta(F)$ drops out. The so-called It\^{o} correction 
term arises from the last term; it amounts to the natural
pairing of $\nabla^2\zeta(F)$, a bilinear form on the ambient space ${\rm End}(\mbox{cotangent space})$, and
thus an element of $( {\rm End}(\mbox{cotangent space}) \otimes {\rm End}(\mbox{cotangent space}) )^\dagger$, and the covariance,
which is an element of the tensor product $ {\rm End}(\mbox{cotangent space}) \otimes {\rm End}(\mbox{cotangent space}) $. In terms of (\ref{ao39}), this
can be expressed as
%
\begin{align*}
\mbox{It\^{o} correction for $\zeta(F)$}=
\frac{1}{2}\int \mu(dE)\,\nabla^2\zeta(F).(FE,FE).
\end{align*}
%
When considering $ \zeta ( F ) = | F |^2 $, it is convenient to write $\zeta(F)$ $={\rm tr}F^*F$, where $F^*\in{\rm End}(\mbox{cotangent space})$ is the adjoint of $F$. Hence the above equation turns into
\begin{align}\label{ao50}
\mbox{It\^{o} correction for $|F|^2$}= \int \mu(dE)\,{\rm tr}(FE)^*FE
={\rm tr}F^*F\int \mu(dE)\,EE^*.
\end{align}
It follows from (\ref{ao42}), since $E_1,E_2,E_3 \in\textbf{O}(2)$, and (\ref{ao41})
\begin{align}\label{ao50b}
\mbox{It\^{o} correction for $|F|^2$}:= (2\kappa_{sym}+\kappa_{skew})|F|^2.
\end{align}
Hence we learn that provided the equivalence of Itô and Stratonovich in \eqref{equiv01} holds, then $ 2\kappa_{sym}+\kappa_{skew}=1 $ holds iff $d\mathbb{E}|F|^2=\mathbb{E}|F|^2d\tau$, which in view of
$\mathbb{E}|F_{\tau=0}|^2$ $=|{\rm id}|^2$ $=2$ is equivalent to (\ref{ao35}).
%


\subsection{Intermittent behavior of the Bessel-type process $R:=\frac{1}{2}|F|^2$,
connection to geometric Brownian motion} \label{intermittent}

We shall monitor the Frobenius norm
\begin{align}\label{ao58prev}
R:=\frac{1}{2}|F|^2\ge 1,
\end{align}
where the inequality is a consequence of ${\rm det}F=1$. We will argue that definition \eqref{ao58prev} is consistent in law with the It\^{o} evolution
\begin{align}\label{ao58}
dR=Rd\tau+\sqrt{R^2-1}dw
\quad\mbox{with}\quad R_{\tau=0}=1,
\end{align}
where $w$ denotes one-dimensional Brownian motion.
This can be seen as a non-Abelian/ hyperbolic version of the fact that the (squared)
Euclidean norm of a multi-dimensional Brownian motion satisfies its own
stochastic evolution, known as the Bessel process.
In fact, 
we shall compare a monotone transformation of the $[1,\infty)$-valued process $R$
to the $[0,\infty)$-valued Bessel process $X$ that corresponds to 
the Euclidean norm of the two-dimensional\footnote{
the connection of the {\it three}-dimensional
$F$ to the {\it two}-dimensional Brownian motion should not be surprising in view of
the fact that $R=1$ corresponds to the {\it one}-dimensional set $F\in\mathbf{SO}(2)$} 
Brownian motion, as defined through 
\begin{align}\label{ao74}
dX=\frac{1}{2}\frac{1}{X}d\tau+dw\quad\mbox{with}\quad X_{\tau=0}=0\quad
\mbox{and satisfying}\quad X_\tau>0\quad\mbox{for}\quad\tau>0,
\end{align}
where the latter is a consequence from the (borderline) transience of 
two-dimensional Brownian motion. From the comparison, we will learn
\begin{align}\label{ao69}
R_\tau>1\quad\mbox{for}\quad\tau>0.
\end{align}
The justification of \eqref{ao69} will be given around equation \eqref{ao65} 
further below in this subsection.

\medskip

We now argue that (\ref{ao58}) holds
for (\ref{ao58prev}). Inserting (\ref{sde01}) into \eqref{ao50b}, we infer from
(\ref{equiv01}) the Itô equation for $R=\frac{1}{2}{\rm tr}F^*F$
\begin{align}\label{sde02}
	d R = R d \tau + { \rm tr } F^* F d B .
\end{align}
Applying the same rationale for deriving the It\^{o} correction term on the level of 
(\ref{sde02}), we learn that for an arbitrary observable $\zeta=\zeta(R)$
%
\begin{align}\label{ao81}
d\mathbb{E}\zeta(R)=\Big(\mathbb{E}R\frac{d\zeta}{dR}
+\frac{1}{2}\mathbb{E}\frac{d^2\zeta}{dR^2}\int\mu(dE)({\rm tr}F^*FE)^2\Big)d\tau.
\end{align}
Hence once we establish the formula
\begin{align}\label{ao80}
\int\mu(dE)({\rm tr}GE)^2= \frac{1}{4} ({\rm tr}G)^2-1
\quad\mbox{for symmetric}\;G\in\mathbf{SO}(2),
\end{align}
we see the r.~h.~s.~of (\ref{ao81}) is given by the expectation of 
$(R\frac{d}{dR}+\frac{1}{2}(R^2-1)\frac{d^2}{dR^2})\zeta$,
which indeed is the generator of (\ref{ao58}).
By \eqref{ao41} and \eqref{sde01}, identity (\ref{ao80}) turns into
%
%
\begin{align*}
({\rm tr}GE_1)^2+({\rm tr}GE_2)^2=({\rm tr}G)^2-4{\rm det}G
\quad { \rm for ~ symmetric } ~ G,
\end{align*}
which is checked by an explicit calculation based on the definition (\ref{ao42}) of $E_1$ and $E_2$.

\medskip

We now start comparing $ R $ to the geometric Brownian motion $ S $ and the Bessel process $ X $. Since the It\^{o}-Stratonovich correction term is given by 
$\frac{1}{2}\sqrt{R^2-1}\frac{d}{dR}\sqrt{R^2-1}$ $=\frac{1}{2}R$,
the Stratonovich form of (\ref{ao58}) reads
\begin{align}\label{ao63}
dR=\frac{1}{2}Rd\tau+\sqrt{R^2-1}\circ dw\quad\mbox{with}\;R_{\tau=0}=1.
\end{align}
Intuitively, to be made precise below, 
this process shares for $R\gg 1$ similarities with the geometric Brownian motion
\begin{align}\label{ao59}
dS=\frac{1}{2}Sd\tau+S\circ dw\quad\mbox{with}\;S_{\tau=0}=1.
\end{align}
By the properties of the Stratonovich form (\ref{ao59}), it can be reformulated as
\begin{align}\label{ao75}
d\ln S=\frac{1}{2}d\tau+dw\quad\mbox{with}\;\ln S_{\tau=0}=1,
\end{align}
so that the solution is given by
\begin{align}\label{ao60}
S_\tau=e^{\frac{1}{2}\tau+w_\tau}.
\end{align}

\medskip

It is well-known that the process $S$ displays intermittency: On the one hand, since Brownian motion grows sublinearly
we read off (\ref{ao60}) that
\begin{align}\label{ao68}
S_\tau\le e^{\alpha\tau}\quad\mbox{for any $\alpha>\frac{1}{2}$ almost surely for}
\;\tau\uparrow\infty.
\end{align}
On the other hand, since the It\^{o} form of (\ref{ao59}) is given by
\begin{align}\label{ao61}
dS=Sd\tau+S dw\quad\mbox{with}\;S_{ \tau=0 }=1,
\end{align}
and the last term $Sdw$ is a martingale, we learn that $d\mathbb{E}S=\mathbb{E}Sd\tau$ so that
\begin{align}\label{ao66}
\mathbb{E}S_\tau=e^\tau\stackrel{(\ref{ao35})}{=}\mathbb{E}R_\tau,
\end{align}
From comparing (\ref{ao68}) for some $\alpha<1$ and (\ref{ao66}) 
we learn that $ e^{ - \tau } S_\tau$ cannot be uniformly integrable as $\tau\uparrow\infty$.

\medskip

For later purpose, we quantify this non-uniform integrability.
We claim that (\ref{ao60}) and (\ref{ao66}) imply that high super level sets of $S$
contain already half of its mass:
\begin{align}\label{ao67}
\mathbb{E}SI(S\ge \mathbb{E}^\frac{3}{2}S)=\frac{1}{2}\mathbb{E}S.
\end{align}
Indeed, by (\ref{ao60}) and (\ref{ao66}) this is equivalent to
\begin{align*}
e^{-\frac{\tau}{2}}\mathbb{E}e^{w_\tau}I(w_\tau\ge\tau)=\frac{1}{2}.
\end{align*}
Since $w_\tau$ is a Gaussian of vanishing mean and variance $\tau$, this in turn is equivalent
to the elementary identity $\int_\tau^\infty dw \, e^{-\frac{\tau}{2}+w-\frac{w^2}{2\tau}}$ 
$=\frac{1}{2}\int_{-\infty}^\infty dw' \, e^{-\frac{{w'}^2}{2\tau}}$.

\medskip

We now consider higher moments of $S$ and $R$. To treat both the same way,
we appeal to the It\^{o} form of the respective evolutions, which has the advantage of isolating a martingale term, at the expense of an
It\^{o} correction term. In case of $S$, it
follows from (\ref{ao61}) that
$d\mathbb{E}S^p$ $=p\mathbb{E}S^pd\tau$ $+\frac{p(p-1)}{2}\mathbb{E}S^pd\tau$,
where the last term is the It\^{o} correction. Hence (\ref{ao66}) generalizes to
\begin{align}\label{ao78}
\mathbb{E}S_\tau^p=e^{\frac{p(p+1)}{2}\tau}\stackrel{(\ref{ao66})}{=}
(\mathbb{E}S_\tau)^{\frac{p(p+1)}{2}},
\end{align}
which reveals a strongly non-Gaussian behavior of moments.
On the one hand, based on the
It\^{o} form (\ref{ao58}), and using that $R^2-1\le R^2$, we obtain the
above differential equality for $\mathbb{E}S^p$ in form of a differential
inequality for $\mathbb{E}R^p$, namely
$d\mathbb{E}R^p$ $\le p\mathbb{E}R^pd\tau$ $+\frac{p(p-1)}{2}\mathbb{E}R^pd\tau$.
Hence by integration
\begin{align*}
\mathbb{E}R_\tau^p\le
\mathbb{E}S_\tau^p.
\end{align*}
%
%

\medskip

We now turn to the comparison of $R$ with the Bessel process $X$ and
the geometric Brownian motion $S$ from below. 
To this purpose, we seek a transformation of $[1,\infty)$ such that
the diffusion part of (\ref{ao63}) turns into the one of (\ref{ao59}). Thanks to
the Stratonovich form, this can be reduced to the chain rule and thus 
amount to the differential relation
\begin{align}\label{ao65}
dS=\frac{S}{\sqrt{R^2-1}}dR\quad\mbox{and}\quad S=1\;\mbox{for}\;R=1,
\end{align}
which has the explicit solution
\begin{align}\label{ao73}
S(R)=\exp{\rm arccosh}\,R
\quad\mbox{with inverse}\quad R(S)=\cosh\,\ln S=\frac{1}{2}(S+\frac{1}{S}).
\end{align}
Incidentally, it follows from the last identity that
$S$ is the larger of the two eigenvalues of the symmetric endomorphism $F^*F$.
By \eqref{ao65} and
$\frac{R}{\sqrt{R^2-1}}=\frac{S^2+1}{S^2-1}$, cf.~(\ref{ao73}), we obtain
\begin{align}\label{ao64}
\tilde S_\tau:=S(R_\tau)\quad\mbox{satisfies}\quad
d\tilde S=\frac{1}{2}\tilde S \frac{\tilde S^2+1}{\tilde S^2-1} d\tau+\tilde S \circ dw
\quad\mbox{with}\quad\tilde S_{\tau=0}=1,
\end{align}
which we write as
\begin{align}\label{ao72}
d\ln\tilde S=\frac{1}{2}\frac{\tilde S^2+1}{\tilde S^2-1}d\tau+dw
\quad\mbox{with}\quad\ln\tilde S_{\tau=0}=0.
\end{align}
The elementary inequalities
\begin{align*}
\frac{S^2+1}{S^2-1}\ge\frac{1}{\ln S}\quad\mbox{and}\quad
\frac{S^2+1}{S^2-1}\ge1\quad\mbox{both for}\;S>1
\end{align*}
allow us to compare (\ref{ao72}) with (\ref{ao74}) and with (\ref{ao75}).
Indeed, since the Stratonovich notion arises from smooth
approximation of $w$, we may use standard ODE arguments to conclude
%
\begin{align}\label{aonew}
\ln\tilde S_\tau\ge X_\tau\quad\mbox{and}\quad\tilde S_\tau\ge S_\tau.
\end{align}
Presumably, the second inequality in (\ref{aonew}) 
and thus (\ref{ao62}) can also be inferred\footnote{
As can be seen from (\ref{ao01bis}), the Gram matrix $G:=F^*F$
satisfies its own evolution, namely $dG=G\circ dB+dB^*\circ G$. 
In \cite[Subsection 7.6.3]{LeJan}, this symmetric matrix $G$ with ${\rm det}G=1$ (or rather its
square root, which arises in the polar factorization of $F$) is parameterized  as 
$\left(\begin{array}{rr}S&Y\\Y&\frac{Y^2+1}{S}\end{array}\right)$, where $S$ is 
(in law) given by (\ref{ao60}) and the law of $Y$ is also characterized.
This implies $R=\frac{1}{2}{\rm tr}G\ge\frac{1}{2}(S+\frac{1}{S})$, in line with
the second items in (\ref{ao73}) and (\ref{aonew}).}
from \cite[Subsection 7.6.3]{LeJan}.
From the first item, we learn that (\ref{ao74}) translates into (\ref{ao69}).
Because of the elementary inequality $S(R) \le 2R$, cf.~(\ref{ao73}), the second item yields
\begin{align}\label{ao62}
2R_\tau\ge S_\tau.
\end{align}
%

\medskip

For later purpose, we note that because of (\ref{ao62}), (\ref{ao66}) and (\ref{ao67})
imply
\begin{align}\label{ao79}
\mathbb{E}RI(R\ge \frac{1}{2}\mathbb{E}^\frac{3}{2}R)\ge\frac{1}{4}\mathbb{E}R.
\end{align}
As for $S$, very high super level sets of $R$ already contain a fraction of the total mass.

\subsection{The main result relating $u$ to $F$ and consequences}
We specify a class of initial value problems for the It\^{o} evolution
\begin{align} \label{ao01}
dF=F dB ,
\end{align}
which by \eqref{equiv01} is equivalent to \eqref{ao01bis}.
Denoting by $\tau$ the independent variable in (\ref{ao01}),
for $\tau_*,\tau\ge 0$, we define $F_{\tau_*,\tau}$ as
\begin{align}\label{ao29}
[\tau_*,\infty)\ni\tau\mapsto dF_{\tau_*,\tau}
\quad\mbox{solves (\ref{ao01}) with initial data}\quad F_{\tau_*,\tau=\tau_*}={\rm id},\nonumber\\
\mbox{extended by}\quad F_{\tau_*,\tau}={\rm id}\quad\mbox{for}\;\tau\le \tau_*.
\end{align}
Equipped with this notation, we may formulate our main result;
it shows that the effect of the quenched/environmental noise $b$ 
can be well-represented by a thermal noise $B$:

\begin{theorem}\label{thm:mr01}
There exists a coupling of $b$ and $B$ such that for all $(x,t)$
\begin{align}\label{ao03}
\mathbb{E}\fint_0^Tdt \,
\frac{1}{|x|^2}|u(x,t)-u(0,t)-F_{\tau(|x|^2),\tau(T)}^\dagger x|^2
\lesssim\varepsilon^2\mathbb{E}|F_{0,\tau(T)}|^2 , 
\end{align}
where
\begin{align}\label{ao12}
\tau(s):=\ln\sqrt{1+{\textstyle\frac{\varepsilon^2}{2}}\ln(1+s)}.
\end{align}
\end{theorem}

Here and in the sequel, the statement $A\lesssim B$ means that there exists
a constant $C$ such that $A\le C B$.

\medskip

Note that the transposition $F^\dagger$ is a collateral damage of
our definition of the gradient, see the discussion in Subsection \ref{ss:X}. 
We stress the fact that the same coupling works for all space-time points $(x,t)$.
In particular, from letting $|x|\downarrow 0$ in (\ref{ao03}) it follows that
\begin{align}\label{ao07}
\mathbb{E}\fint_0^Tdt|\nabla u(0,t)-F_{0,\tau(T)}|^2
\lesssim\varepsilon^2\mathbb{E}|F_{0,\tau(T)}|^2,
\end{align}
so that $\nabla u(0,t)\approx F_{0,\tau(T)}$ for $\varepsilon\ll 1$ on average in $t\in(0,T)$, which is implicitly contained in \cite[Section 3.3: (52) onwards]{OW24}.
Obviously, (\ref{ao07}) cannot hold uniformly in time since $\nabla u(0,0)$ $={\rm id}$.

\medskip

Note that in view of (\ref{ao07}),  
(\ref{ao03}) seems to take the form of a Taylor remainder estimate.
However, because of the subtle dependence on the radial variable $|x|$ via $\tau(|x|^2)$, 
this is only true w.~r.~t.~the angular dependence on $\frac{x}{|x|}$,
or on average over scales of a given order. Roughly speaking, this suggests a picture where, at very large scales, $u$ is well-approximated by an affine function, the slope of which, however, varies with the scale.  The same informal description has been posited in \cite{ABK24}; see the discussion preceding Conjecture E therein.

\medskip

Note that by homogeneity of (\ref{ao01}) in $\tau$, (\ref{equiv02}) also characterizes the
law of $F_{\tau(|x|^2),\tau(T)}$ appearing on the l.~h.~s.~of (\ref{ao03}).
As a consequence, we obtain

\begin{corollary}\label{Cor:1}
In the regime
\begin{align}\label{ao09}
\varepsilon^2 \lambda(|x|^2)\ll 1
\end{align}
we have
\begin{align}\label{ao06}
\mathbb{E}\fint_0^Tdt\frac{1}{|x|^2}|u(x,t)-u(0,t)|^2\approx
\max\Big\{1,\frac{\lambda(T)}{\lambda(|x|^2)}\Big\},
\end{align}
where
\begin{align}\label{ao11}
\lambda(s):=\sqrt{1+{\textstyle\frac{\varepsilon^2}{2}}\ln(1+s)}
\stackrel{(\ref{ao12}),\eqref{ao35}}{=}\frac{1}{|x|^2}
\mathbb{E}|F_{0,\tau(s)}^{\dagger}x|^2.
\end{align}
\end{corollary}

Here and in the sequel, a statement of the form ($A\ll B$ $\Longrightarrow$ 
$A'\approx B'$) means that for every $\varepsilon>0$ there exists a $\delta>0$ such that
($A\le\delta B$ $\Longrightarrow$ 
$\frac{1}{1+\varepsilon}B'\le A' \le(1+\varepsilon)B'$).
We learn from (\ref{ao11}) that in our regime of (\ref{ao09}), which ensures a small
noise-to-signal ratio in (\ref{ao03}), a significant dependence on $(x,t)$ is only seen 
on scales that are exponentially large in terms of $\varepsilon$.
From letting $|x|\downarrow 0$ in (\ref{ao06}) it follows in particular
\begin{align*}
\mathbb{E}\fint_0^Tdt|\nabla u(0,t)|^2
\approx 2\lambda(T)\quad\mbox{for}\;\varepsilon\ll 1.
\end{align*}
Note that in the regime $1\ll\lambda(|x|^2)\ll\lambda(T)$ (which implies $1\ll |x|^2\ll T$), 
the r.~h.~s.~of (\ref{ao06}) is $\approx\sqrt{\frac{\ln T}{\ln |x|^2}}$.
Hence the transformation $x\mapsto u(t,x)$ is borderline non-Lipschitz in a wide range of $|x|$.
This is in line with our discussion in Subsection \ref{ss:X}.

\medskip

Our last result lifts the intermittency worked out on the level of $F$
in Subsection \ref{intermittent} to the level of $u$.
We express intermittency in terms of an anomalous scaling of higher stochastic moments $p>1$ of 
quadratic increments of $u$; the exponent $1+\frac{3}{2}(p-1)>p$ in (\ref{ao77})
reveals that the modulus of continuity, more precisely its logarithmic correction, depends on $p$.
We note that the identity (\ref{ao78}) rather suggests the exponent 
$\frac{p(p+1)}{2}$ $>1+\frac{3}{2}(p-1)$ in (\ref{ao77}). 
However, while we can transmit (\ref{ao78}) from $S$ to $R=\frac{1}{2}|F|^2$, 
we are presently not able to transmit it to $|\nabla u|^2$ because our homogenization
result is limited to $p=1$, see Theorem \ref{thm:2}.
In fact, we rather transfer (\ref{ao79}), see \eqref{meet01}.
By a simple post-processing, this leads to (\ref{ao77}), where
$1+\frac{3}{2}(p-1)$ is the first-order Taylor
polynomial of $\frac{p(p+1)}{2}$ at $p=1$.

\begin{proposition}\label{prop:mr02}
For every $1\le p<\infty$ we have in the regime (\ref{ao09}) 
\begin{align}\label{ao77}
\mathbb{E} \Big(\fint_0^Tdt\frac{1}{|x|^2}|u(x,t)-u(0,t)|^2 \Big)^p
\gtrsim_{ p } \max\Big\{1,\frac{\lambda(T)}{\lambda(|x|^2)}\Big\}^{1+\frac{3}{2}(p-1)}.
\end{align}
\end{proposition}

Here and in the sequel, the subscript $\lesssim_\alpha$ indicates
that the implicit multiplicative constant depends only on $\alpha$.

\medskip

\subsection{Incremental homogenization in large-scale cut-off $L$ 
and proxy corrector $\tilde\phi_L$}
The proof of Theorem \ref{thm:mr01} relies on a homogenization result
from \cite{CMOW}. There, it is used to
derive the super-diffusive asymptotics for the annealed second moment of (\ref{ao02}); prior results
relied on a Fock-space analysis of the non-normal generator of the process $b(\cdot+X_t)$
(the environment as seen from the particle), see \cite{TothValko,CannizaroHaunschmidtSibitzToninelli}. 
Subsequently, homogenization theory was also used to sharpen this result to quenched
moment \cite{ABK24}.
We motivate and state the result from \cite{CMOW}
in the form it is presented in \cite{MOW24}.
We start by writing 
\begin{align*}
u(x,t)=x+\phi(x,t),
\end{align*}
so that (\ref{ao04}) turns into
\begin{align}\label{ao14}
\partial_t\phi-b.\nabla \phi-\triangle \phi=b,\quad \phi(t=0)=0;
\end{align}
it has the advantage that $x\mapsto\phi(x,t)$ inherits the stationarity of $b$.
In view of the parabolic character of (\ref{ao14}), it is plausible that
up to time $T$, the solution $\phi$ explores the drift $b$ only
on length scales $L$ $\lesssim\sqrt{T}$.
We thus introduce the stationary and centered vector field $b_L$ via cut-off on the Fourier level
\begin{align}\label{ao32}
{\mathcal F}b_L(k)=I(L|k|\ge 1){\mathcal F}b(k)
\end{align}
and consider (\ref{ao14}) with $b$ replaced by $b_L$ and (formally) pass to $t\uparrow\infty$
\begin{align}\label{ao15}
-b_L.\nabla \phi_L-\triangle \phi_L=b_L.
\end{align}
Equation (\ref{ao15}) is known to define a stationary and centered gradient field $\nabla\phi_L$,
the corrector in the jargon of stochastic homogenization. Note that in view of (\ref{ao24}),
\begin{align}\label{ao25}
u_L(x):=x+\phi_L(x)\quad\mbox{componentwise satisfies}\quad\nabla.(u_L^ib_L+\nabla^* u_L^i)=0,
\end{align}
where the transpose $\nabla^* u$ is the gradient (tangent) vector field,
and the upper index denotes a Cartesian component of a tangent vector.
Under the coordinates $u_L$, the process $X_L$, 
defined through (\ref{ao02}) with $b$ replaced by $b_L$, turns into a martingale. 
By a classical result\footnote{see e.~g.~the textbook \cite[Chapter 11]{KLO}},
this martingale behaves like a Brownian motion on large space-time scales.
Moreover, the covariance tensor of this Brownian motion in direction $i$
is given by twice the expectation 
$ \E u_{L}^ib_L+\nabla^* u_{L}^i$ of the flux
$u_{L}^{i} b_L+\nabla^* u_{L}^{i}$, cf.~(\ref{ao25}).
In view of the rotational invariance of the ensemble of $b_L$'s,
the covariance tensor has the same property and is 
thus characterized by a scalar $\lambda_L$.
Hence a calculation based on $\mathbb{E}
(u_{L}^{i} b_L+\nabla^* u_{L}^{i}).\nabla\phi_L^{i}$ $=0$, which follows from (\ref{ao25}),
yields the representation 
\begin{align}\label{ao28}
\lambda_L=\mathbb{E}|\nabla u_L^i|^2\quad\mbox{and thus}\quad 2\lambda_L=\mathbb{E}|\nabla u_L|^2,
\end{align}
where by stationarity, it does not matter at which space point $x$ we evaluate $\nabla u_L^i$,
so that we dropped the argument. 
%
%
On the level of the generator of the process, stochastic homogenization means
\begin{align}\label{ao16}
b_L.\nabla+\triangle\approx\lambda_L\triangle,\quad\mbox{loosely speaking}.
\end{align}
We let ourselves guide by (\ref{ao16}) to define a tractable approximation to $\phi_L$
via a scale-by-scale homogenization: We take the derivative of (\ref{ao15}) with respect to $L$,
apply Leibniz' rule and perform two approximations: 
\begin{itemize}
\item Homogenization: Appealing to (\ref{ao16}), we replace 
$-b_L.\nabla d\phi-\triangle d\phi$ by $-\tilde\lambda_L\triangle d\phi$,
for a self-consistent choice choice of $\tilde\lambda_L$, see (\ref{ao22}).
\item Linearization: Retaining just the first-order terms in $\varepsilon$,
we ignore $-db.\nabla\phi_L$.
\end{itemize}
Summing up, we define the stationary centered scalar field $\phi_L$ 
by integration in $L\in[1,\infty)$ of
\begin{align}\label{ao17}
- \tilde\lambda_L\triangle d\phi=d b\quad\mbox{with}\quad\phi_{L=1}=0.
\end{align}
%
The approximation defined through (\ref{ao17}) is not good enough and needs to be post-processed.
The post-processing relies on the following refinement of (\ref{ao16})
\begin{align}\label{ao26}
(b_L.\nabla+\triangle)(1+\phi_L^i\partial_i)\approx\lambda_L\triangle,
\quad\mbox{loosely speaking}
\end{align}
where lowercase indices denote the Cartesian coordinate of a cotangent vector
and we use Einstein's convention of summation over repeated indices,
provided one is uppercase and the other lowercase. What is known as the
two-scale expansion in homogenization, namely the operator $1+\phi_L^i\partial_i$,
intertwines the heterogeneous generator $b_L.\nabla+\triangle$ with the homogeneous
one $\lambda_L\triangle$.
Hence we define $\tilde\phi_L$ via the stochastic differential equation
\begin{align}\label{ao18}
d\tilde\phi=(1+\tilde\phi_L^i\partial_i)d\phi\quad\mbox{with}\quad\tilde\phi_{L=1}=0,
\end{align}
which has to be interpreted in the It\^{o} sense, so that $\tilde\phi$ is a centered
martingale, with values in the space of stationary fields. Returning to
\begin{align*}
\tilde u_L(x,t):=x+\tilde\phi_L(x,t),
\end{align*}
we indeed obtain a good approximation in the regime $\varepsilon\ll 1$:

\begin{theorem}[{\cite[Lemma 7 \& (100)]{CMOW} \& \cite[Proposition 1.1]{MOW24}}]\label{thm:2}
Provided 
\begin{align}\label{ao22}
L=\sqrt{T + 1}\quad\mbox{and}\quad \tilde\lambda_L=\lambda(T) ~ { \rm in } ~ \eqref{ao17}
\end{align}
we have
\begin{align*}
\fint_0^Tdt\mathbb{E}|\nabla u(0,t)-\nabla\tilde u_L(0)|^2\lesssim\varepsilon^2\lambda(T).
\end{align*}
\end{theorem}

In view of the inequality $\mathbb{E}|u(x)-u(0)|^2$ $\le|x|^2\mathbb{E}|\nabla u|^2$
for any $u$ with stationary $\nabla u$ we obtain

\begin{corollary}\label{cor:2}
In the setting of Theorem \ref{thm:2} we have
\begin{align*}
\fint_0^Tdt\frac{1}{|x|^2}
\mathbb{E}|(u(x,t)-u(0,t))-(\tilde u_L(x)-\tilde u_L(0))|^2\lesssim\varepsilon^2\lambda(T).
\end{align*}
\end{corollary}


\subsection{Coupling $B$ to $b$}\label{ss:coupl}
We are now in the position to define the coupling of $b$ and $B$ of Theorem \ref{thm:mr01}.
In view of Theorem \ref{thm:2}, it amounts to connect $\nabla \tilde{u}_L(0)$ 
to $F_{0,\tau}$ defined through (\ref{ao29}).
Applying $\nabla$ to (\ref{ao18}) we obtain the It\^{o} equation
\begin{align}\label{ao20}
d\nabla\tilde u=(\nabla\tilde u+\tilde\phi_L^i\partial_i)\nabla d\phi\quad\mbox{with}\quad
\nabla\tilde u_{L=1}={\rm id}.
\end{align}
The last approximation is to neglect the second r.~h.~s.~term in (\ref{ao20}).
Incidentally,
since the driver $\nabla d\phi$ can be seen to be log-log correlated,
the field $\nabla \tilde u$ may be assimilated to a tensorial version of the Gaussian
multiplicative chaos. Evaluating (\ref{ao20}) at $x=0$ we thus expect $\nabla\tilde u_{L}(0)\approx F$ where $F$ 
satisfies the endomorphism-valued It\^{o} evolution equation
\begin{align}\label{ao21}
dF=F\nabla d\phi(0).
\end{align}
We claim that (\ref{ao21}) is identical to (\ref{ao01}) in law, 
provided we relate the independent variables $L\in[1,\infty)$
and $\tau\in[0,\infty)$ according to
\begin{align}\label{ao27}
e^\tau=\lambda(L^2-1)\stackrel{\varepsilon\ll1}{\approx}\lambda(L^2).
\end{align}
Note that this relationship between $\tau$ and $\lambda$ is consistent 
with the combination of (\ref{ao12}) and (\ref{ao11}).
This identity in law shows that the choice (\ref{ao22}) was self-consistent:
If $F_L$ denotes the solution of (\ref{ao21}) with initial data $F_{L=1}={\rm id}$,
the laws of $F_L$ and $F_{0,\tau}$ are identical.
By relation (\ref{ao27}),
identity (\ref{ao35}) then turns into $\mathbb{E}|F_{L}|^2$ 
$=2\lambda(L^2 - 1)$
$\approx2\lambda(L^2)$.
This is indeed consistent with (\ref{ao28}), if one believes in the approximations
informally discussed above, namely, that $(\nabla u_L,\lambda_L)$ 
is approximated by $(\nabla\tilde u_L,\tilde\lambda_L)$,
and that $\nabla \tilde{u}_{L}(0)$ is approximated by the solution $F_L$ of (\ref{ao21})
with initial value $F_{L=1}={\rm id}$.

\medskip

We now establish this identity in law of (\ref{ao01}) and (\ref{ao21}),
which amounts to showing that the driver $\nabla d\phi(0)$ in (\ref{ao21}), 
after the change of independent variables according to (\ref{ao27}), has the same law as $B$. 
We define the Gaussian process $B$ according to
\begin{align}\label{ao37bis}
B_L:=\frac{\sqrt{2}}{\varepsilon}\nabla(-\triangle)^{-1}b_L(0).
\end{align}
In view of (\ref{ao24}), we have $B_L\in\mathfrak{sl}(2)$.
In view (\ref{ao32}) and the white-noise character of $b$, cf.~(\ref{ao08}),
$L\mapsto B_L$ has independent increments. 
It remains to identify the correlation tensor of $B_L$, 
as an element of ${\rm End}(\mbox{cotangent space})\otimes{\rm End}(\mbox{cotangent space})$. 
We note that by (\ref{ao08}) and (\ref{ao32}), the Fourier transform of 
the tensorial covariance function of $\frac{\sqrt{2}}{\varepsilon}
\nabla(-\triangle)^{-1}b_L$ is given by
\begin{align*}
\frac{2}{|k|^2}I(L^{-1}<|k|\le 1)\big(
\sum_{i=1}^2\frac{(e_i\otimes k)\otimes(e_i\otimes k)}{|k|^2}
-\frac{(k^*\otimes k)\otimes(k^*\otimes k)}{|k|^2}\big),
\end{align*}
where we used that ${\rm End}(\mbox{cotangent space})$ is canonically isomorphic to the tensor product
$(\mbox{tangent space}) \otimes(\mbox{cotangent space})$. We recover
the covariance tensor of the random variable $B_L=\frac{\sqrt{2}}{\varepsilon}
\nabla(-\triangle)^{-1}b_L(0)$ by applying $\int\frac{dk}{2\pi}$ to this identity,
which appealing to polar coordinates yields
\begin{align}\label{ao52}
\mathbb{E} B_L\otimes B_L
=2(\ln L)\fint_{S^1}dk\sum_{i=1}^2(e_i\otimes k)\otimes(e_i\otimes k)
-(k^*\otimes k)\otimes(k^*\otimes k).
\end{align}
From the $\textbf{SO}(2)$-invariance and $L$-dependence of this expression
we deduce that 
\begin{align}\label{ao36}
[0,\infty)\ni \ln L\mapsto B
\in\mathfrak{sl}(2)\quad\mbox{is an $\textbf{SO}(2)$-invariant Brownian motion}.
\end{align}

\medskip

Definition (\ref{ao17}) of $\phi_L$,
together with the choice (\ref{ao22}) of $\tilde\lambda_L$, 
and definition (\ref{ao37bis}) of $B_L$ imply
\begin{align}\label{ao37}
\frac{\sqrt{2}\lambda(L^2-1)}{\varepsilon}\nabla d\phi(0)=dB.
\end{align}
The definition of the change of variables (\ref{ao27})
and the definition (\ref{ao11}) of $\lambda(s)$ imply
\begin{align*}
(\frac{\sqrt{2}\lambda(L^2-1)}{\varepsilon})^2d\tau=d\ln L.
\end{align*}
%
Hence by the parabolic scale invariance of Brownian motion, (\ref{ao36}) translates into
the desired
\begin{align*}
[0,\infty)\ni \tau\mapsto \nabla\phi(0)
\in\mathfrak{sl}(2)\quad\mbox{is an $\textbf{SO}(2)$-invariant Brownian motion}.
\end{align*}

\medskip

It remains to argue that the covariance tensor $\mathbb{E}B_{L=e}\otimes B_{L=e}$
satisfies
the two postulates in Subsection \ref{ss:F}. In terms of the representation (\ref{ao39})
of the covariance tensor, in view of (\ref{ao49}) and (\ref{ao50}) we need to check
\begin{align}\label{ao51}
\int\mu(dE) E^2=0\quad\mbox{and}\quad\int\mu(dE)EE^*={\rm id}.
\end{align}
Now the representation (\ref{ao52}) is of the form (\ref{ao39}).
In order to establish the first item in (\ref{ao51}), 
we note that the integrand in (\ref{ao52}), when replacing $ E \otimes E $ by $ E^2 $, vanishes:
\begin{align*}
\sum_{i=1}^2(e_i\otimes k)(e_i\otimes k)
-(k^*\otimes k)(k^*\otimes k)
=\sum_{i=1}^2(k.e_i)e_i\otimes k
-(k^*\otimes k)=0 .
\end{align*}
For the second item in (\ref{ao51}), we need to replace $ E \otimes E $
by $ E E ^* $ in the integrand (\ref{ao39}), which leads to
\begin{align*}
\sum_{i=1}^2 (e_i\otimes k)(k^*\otimes e_i^*)
-(k^*\otimes k)(k^*\otimes k)
= k^* \otimes k  .
\end{align*}
It remains to note that $\fint_{S^1}dk \, k^*\otimes k$ $=\frac{1}{2}{\rm id}$, 
which by $\textbf{SO}(2)$-invariance can be checked by comparing the traces.


\section{Proofs of the Main Results}

The main new contribution of this note is Theorem \ref{thm:mr01}.  It will be proved by combining results of the previous works \cite{CMOW,MOW24,OW24} with the following lemma, which is a similar linear approximation but applied to $\tilde{u}_{L}$ instead of $u(\cdot,t)$.

\begin{lemma}\label{lem:mr03}
We have with the abbreviations $ \tau_* = \tau( | x|^2 ) $ and $ \tilde\lambda_* = \tilde\lambda_{ | x | } $
\begin{align}\label{pmr03}
\frac{1}{|x|^{2}} \E | \tilde u_L ( x ) - \tilde u_L ( 0 ) - F_{ \tau_*, \tau }^{ \dagger } x  |^2
&\lesssim \frac{ \varepsilon^2 }{ \tilde\lambda_{ * }^2 } \Big( 1 + \frac{ \varepsilon^2 }{ \tilde\lambda_{ * }^2 } \Big) \frac{ \tilde\lambda_L }{ \tilde\lambda_{ * } } . 
\end{align}
\end{lemma}

The proof of this lemma will come in Section \ref{sec:mr03} below.  For now, we show how to use it to prove the main results, Theorem \ref{thm:mr01}, Corollary \ref{Cor:1} and Proposition \ref{prop:mr02}.

\medskip

In preparation for the proofs, we recall the following estimate on the proxy corrector $\tilde{\phi}_{L}$:

\begin{lemma} \cite[(31)]{MOW24}\label{proxy_est} We have 
	\begin{align*}
		\mathbb{E} |\tilde{\phi}_{L}|^{2} \lesssim \frac{\varepsilon^{2} L^{2}}{\tilde{\lambda}_{L}^{2}}.
	\end{align*}
\end{lemma}

\begin{proof}[Proof of Theorem \ref{thm:mr01}.] By combining Corollary \ref{cor:2} (recall that $ \tilde\lambda_L = \lambda(T) $ by \eqref{ao22}) with \eqref{pmr03} we estimate
\begin{align}\label{pmr06cont}
\E \fint_0^T dt \, | u ( x, t ) - u ( 0, t ) - F_{ \tau_*, \tau }^{ \dagger } x |^2
\lesssim \varepsilon^2 | x |^2 \Big( \frac{ \tilde\lambda_L }{ \tilde\lambda_{ * } } + \tilde\lambda_L \Big) \quad { \rm provided } ~ | x | \leq L .
\end{align}
In the opposite regime $ L \leq | x | $, we may appeal to Corollary \ref{cor:2} in combination with the bound on $ \tilde\phi $ in the form of the estimate $ \E^{ \frac{1}{2} }  | \tilde \phi_L |^2 \lesssim \varepsilon \frac{ L }{ \tilde\lambda } $ from Lemma \ref{proxy_est}, to conclude
\begin{align}\label{pmr06cont2}
\begin{aligned}
& \E \fint_0^T dt \, | u ( x, t ) - u ( 0, t ) - F_{ \tau_*, \tau }^{ \dagger } x |^2 \\
& \lesssim \varepsilon^2 | x |^2 \tilde\lambda_L  + \E | \tilde \phi_L ( x ) - \tilde \phi_L ( 0 ) |^2
\lesssim \varepsilon^2 | x |^2 \Big( \tilde\lambda_L  + \frac{1}{ \tilde\lambda^2 } \frac{ L^2 }{ | x |^2 } \Big)  \quad { \rm provided } ~ L \leq | x | .
\end{aligned}
\end{align}
In view of $ 2 \tilde\lambda_L = \E | F_{ 0, \tau } |^2 $, see \eqref{ao22} \& \eqref{ao11}, the estimates \eqref{pmr06cont} and \eqref{pmr06cont2} imply \eqref{ao03}.
\end{proof}

We now use Theorems \ref{thm:mr01} and \ref{thm:2} to establish intermittency of the Lagrangian coordinate $u$.  To do so, we will utilize the non-equi-integrability result on the diffusion $F$ from Section \ref{intermittent}.

\begin{proof}[Proof of Proposition \ref{prop:mr02}.]
Let us write $ X = \fint_{S^{1}} d\xi \, \fint_{0}^{T} dt \, \frac{1}{|x|^{2}} |u(|x| \xi,t) - u(0,t)|^{2}$ and $Y = \frac{1}{2} |F_{\tau_{*},\tau}|^{2}$, where here $ d\xi $ is the arc length measure on the unit circle $S^{1}$.  It will be convenient to let $X(\xi) =  \fint_{0}^{T} dt \, \frac{1}{|x|^{2}} |u(|x| \xi,t) - u(0,t)|^{2}$ and $Y(\xi) = |F_{\tau_{*},\tau}^{\dagger} \xi|^{2}$ so that $X = \fint_{S^{1}} d \xi \, X(\xi)$ and likewise for $Y$.  Our first goal is to prove the inequality
\begin{align}\label{pmr02}
\E X I ( X < { \textstyle \frac{1}{4} } r^{ \frac{3}{2} }  ) 
\leq \frac{7}{4} r
\quad { \rm with } \quad
 r = \frac{1}{2} \frac{ \lambda( T ) }{ \lambda( | x |^2 ) } ,
 \end{align}
whenever $ \varepsilon^2 \lambda( | x | ) \ll 1 $ and $\tau - \tau_* \gg 1 $. Recall from Corollary \ref{Cor:1} that asymptotically $ 2r \approx \E X $. In a second step, we post-process this statement into our claim.

\medskip

We start from the observation
\begin{align}\label{pmr02prev}
 X I ( X < { \textstyle \frac{1}{4} } r^{\frac{3}{2}} ) 
\leq Y I ( Y < { \textstyle \frac{1}{2} } r^{\frac{3}{2}} ) + | X - Y | .
\end{align}
On the error term $ | X - Y | $, we first apply Jensen's inequality and the elementary inequality $ || a |^2 - | b |^2|  \leq | a - b | ( 2 | a | + | a - b | )  $ to obtain
\begin{align*}
 \E | X(\xi) - Y(\xi) | 
& \lesssim \Big( \E \fint_0^T dt \, | F_{ \tau_*, \tau } ^{ \dagger }\xi  - |x|^{-1} ( u( |x| \xi, t ) - u ( 0, t ) ) |^2 \Big)^{ \frac{1}{2} } \\
& \times \bigg( \Big( \E \fint_0^T dt \,  | F_{ \tau_*, \tau } |^2 \Big)^{ \frac{1}{2} } + \Big( \E \fint_0^T dt \,  |  F_{ \tau_*, \tau } ^{ \dagger } \xi - |x|^{-1} ( u( |x| \xi, t ) - u ( 0, t ) ) |^2 \Big)^{ \frac{1}{2} } \bigg). 
\end{align*}
Applying Theorem \ref{thm:mr01}, Corollary \ref{Cor:1}, and the bound $\mathbb{E}|X - Y| \leq \fint_{S^{1}} \, d \xi \, \mathbb{E}|X(\xi) - Y(\xi)|$, we find
\begin{align*}
\E | X - Y | 
\lesssim \Big( \varepsilon^2 \frac{\lambda^2(T)}{\lambda(|x|^2)} \Big)^{ \frac{1}{2} } + \varepsilon^{2} \lambda(T)
\ll \frac{1}{2} \frac{\lambda(T)}{\lambda(|x|^2)} = r
\quad { \rm provided} \quad
\varepsilon^2 \lambda( | x |^2 ) \ll 1 .
\end{align*}
Therefore, \eqref{pmr02prev} and \eqref{ao79} together imply
\begin{align*}
\E X I ( X < { \textstyle \frac{1}{4} } r^{ \frac{3}{2} } ) 
\leq \E Y I ( Y < { \textstyle \frac{1}{2} } r^{ \frac{3}{2} } ) + \E | X - Y | 
\leq \frac{7}{4} r
\quad { \rm with } \quad
 r =  \frac{1}{2} \frac{ \lambda( T ) }{ \lambda( | x |^2 ) }
 \end{align*}
whenever $ \varepsilon^2 \lambda( | x | ) \ll 1 $ and $\tau - \tau_* \gg 1$, as claimed in \eqref{pmr02}.

\medskip

To conclude the proposition in case $\tau - \tau_* \gg 1$, let us note that Corollary \ref{Cor:1}, in form of $ \E X \approx 2r  $, and \eqref{pmr02} imply
\begin{align}\label{meet01}
\E X I ( X \geq { \textstyle \frac{1}{4} } r^{ \frac{3}{2} } )
\geq \frac{1}{8} r ,
\end{align}
which by Markov's inequality shows
\begin{align*}
\E X^p 
\geq \frac{1}{8} \Big( \frac{1}{4} r^{ \frac{3}{2} } \Big)^{p - 1} r.
\end{align*}
After applying Jensen's inequality and isotropy in law, this yields
	\begin{align*}
		\mathbb{E} \Big( \fint_{0}^{T} dt \, \frac{1}{|x|^{2}} |u(x,t) - u(0,t)|^{2} \Big)^{p} \geq \E X^{p} \gtrsim r^{1 + \frac{3}{2} (p - 1)}.
	\end{align*}
Recalling again that by Corollary \ref{Cor:1} we have $ 2r \approx \fint_{0}^{T} dt \, \frac{1}{|x|^{2}} |u(x,t) - u(0,t)|^{2}$, this implies the claim in the regime $\tau - \tau_* \gg 1 $.

\medskip

If $ \tau - \tau_* \gg 1 $ does not hold, we may apply Jensen's inequality to \eqref{ao06} to obtain
\begin{align*}
	1 \lesssim \Big( \mathbb{E}\fint_0^Tdt\frac{1}{|x|^2}|u(x,t)-u(0,t)|^2 \Big)^p
	\leq   \mathbb{E} \Big( \fint_0^Tdt\frac{1}{|x|^2}|u(x,t)-u(0,t)|^2 \Big)^p,
\end{align*}
so that the proposition holds trivially in this case.
\end{proof}

\section{Proof of Lemma \ref{lem:mr03}}\label{sec:mr03}

\subsection{Strategy of Proof} Our interest in Lemma \ref{lem:mr03} is in the increments of $\tilde{u}_L $.  Fix $x \in \mathbb{R}^{2}$.  Recall from the introduction that since $ \tilde u_L = x + \tilde\phi_L $, we may use \eqref{ao18} to rewrite
	\begin{align*}
		& d (\tilde{u}(x) - \tilde{u}(0)) \\ 
			&= d \phi(x) + \tilde{\phi}^{i}(x) \partial_{i} d \phi (x) - d \phi ( 0 )- \tilde{\phi}^{i} ( 0 ) \partial_{i} d \phi ( 0 ) \\
			&= ( x^i + \tilde{\phi}^i(x) - \tilde{\phi}^i(0) )  \partial_i d \phi(0)
			+ \tilde\phi^i(x) \partial_i ( d \phi(x) - d \phi(0) )
			+ d ( \phi(x) - \phi(0) - x . \nabla \phi ( 0 ) ) .
	\end{align*}
After regrouping terms, this says that 
	\begin{align} \label{E: equation for increment}
		d(\tilde{u}(x) - \tilde{u}(0)) &= ( \tilde{u} (x) - \tilde{u} (0) ) . \nabla d \phi ( 0 )  + df,
	\end{align}
where $f$ is defined through
	\begin{equation} \label{E: forcing}
		df = \tilde{\phi} (x) . \nabla ( d \phi(x) - d \phi(0) ) + d ( \phi(x) - \phi(0) - x . \nabla \phi(0) ).
	\end{equation}
In view of the equation \eqref{ao01} satisfied by $F$ and \eqref{E: equation for increment}, we may write
	\begin{equation*}
	r_{L} := \tilde{u}_{L}(x) - \tilde{u}_{L}(0) - F_{ \tau_{*}, \tau }^{\dagger} x ,
	\end{equation*}
where the residuum $ r_L $ solves the initial value problem
	\begin{align}
		d r &= r . \nabla d \phi(0) + df, \quad r_{L_{*}} = \tilde{\phi}_{L_{*}}(x) - \tilde{\phi}_{L_{*}}(0)  \label{E: equation of mu}
	\end{align}
and, consistently with the notation for $ \tau_* $ and $ \tilde\lambda_* $, we write $ L_* = | x | $.
In particular, written in terms of $r$, the main result \eqref{pmr03} states that
	\begin{equation} \label{E: main result linearization rewritten}
		 \frac{1}{|x|^{2}}\mathbb{E} | r_{L} |^{2} \lesssim \varepsilon^{2} \frac{\tilde{\lambda}}{\tilde{\lambda}_{*}} .
	\end{equation}
When establishing \eqref{E: main result linearization rewritten}, we rely on the observation that to leading order, the driver in $  r . \nabla d \phi(0) $ in \eqref{E: equation of mu} leads to the differential inequality $ d \E | r |^2 \leq \E | r |^2 d \tau + ... $, which contains \eqref{E: main result linearization rewritten} upon establishing that initially $ \frac{1}{|x|^{2}} \mathbb{E} | r_* |^{2} \lesssim \frac{\varepsilon^{2}}{\tilde{\lambda}_{*}^{2}} $. More specifically, we use the following lemma in terms of quadratic variations.
\begin{lemma}[{\cite[Lemma 4.4 \& (75)]{OW24}}]\label{lem:wr65bis}
	We have\footnote{In the language of \cite[Definition 4.2]{OW24}, we have\
	\begin{align*}
	{ \textstyle r_i r_j [ \partial_i d \phi ( 0 ) \cdot \partial_j d \phi ( 0 ) ] = \sum_{ i = 1 }^{2}  [ e^i \otimes r . \nabla d \phi ( 0 ) \,  e^i \otimes r . \nabla d \phi ( 0 ) ] = \sum_{ i = 1 }^{2} e^i \otimes r \diamond  e^i \otimes r \frac{d \tilde\lambda^2 }{ \tilde\lambda^2 }, }
	\end{align*}
	so that \cite[(75)]{OW24} indeed applies.}
	\begin{align}\label{wr65bis}
		r^i r^j [ \partial_i d \phi \cdot \partial_j d \phi  ] = | r |^2 d \tau .
	\end{align}
\end{lemma}
Note that in \eqref{wr65bis}, we may suppress the spatial argument of $ \nabla d \phi $ since by the independence of $ \nabla d \phi $'s increments and stationarity the quadratic variation in \eqref{wr65bis} is constant (w.~r.~t.~the spatial variable) and deterministic.

\medskip

Note that in this note, \eqref{wr65bis} may easily be derived from the second item in \eqref{ao51}. Indeed, we may use definition \eqref{ao37} to infer that $ r^i  \partial_i d \phi $ is identical in law to $ d B^{ \dagger } r $ modulo the change of variables \eqref{ao27}. Since $ B $ has independent increments, this shows that the l.~h.~s.~of~\eqref{wr65bis} is nothing but $ d \E | B^{ \dagger } r |^2 $, so that Lemma \ref{lem:wr65bis} reduces to understanding the evolution of the covariance of $ B^{ \dagger } $ as in \eqref{ao39}. For our purposes, the second item in \eqref{ao51} is sufficient to conclude
\begin{align*}
r^i r^j [ \partial_i d \phi \cdot \partial_j d \phi  ]  
= \int \mu( dE ) ( E E^* )^{ \dagger } r \cdot r
\stackrel{ \eqref{ao51} }{=} | r |^2 d \tau ,
\end{align*}
as claimed in Lemma \ref{lem:wr65bis}.

\subsection{Proof of Lemma \ref{lem:mr03}}

To obtain this estimate, we will utilize estimates of covariation terms coming from the driver $\nabla d\phi$ and the forcing $df$. Toward that end, in Appendix \ref{sec:appendix-quad-var} below, we prove that\footnote{The following inequalities have to be understood in the sense of measures, i.e. $ | \mu | \lesssim \nu \Leftrightarrow \int f d | \mu | \leq \int f d \nu $ for any non-negative (measurable) function $ f $.}
	\begin{align}
		| [\partial_{i} d\phi (0) \cdot (\partial_{j}d\phi (x) - \partial_{j}d\phi (0) )] | &\lesssim | x | \frac{d \tau }{L}, \label{E: covariation estimate 1} \\
		|[  (\partial_{i}d\phi (x) - \partial_id\phi(0) ) \cdot (\partial_{j}d\phi (x) - \partial_{j}d\phi (0) )]| &\lesssim | x |^2 \frac{d \tau}{L^2}, \label{E: covariation estimate 1 b} \\
		| [\partial_{i} d\phi (0) \cdot (d\phi (x) - d\phi (0) - x .\nabla d\phi (0))] | &\lesssim | x |^{2} \frac{d \tau}{L} \label{E: covariation estimate 2}, \\
		[  (d\phi (x) - d\phi (0) - x .\nabla d\phi (0)) \cdot (d\phi (x) - d\phi (0) - x .\nabla d\phi (0))] &\lesssim | x |^{4} \frac{d \tau}{L^2}. \label{E: covariation estimate 2 b}
	\end{align}

\begin{proof}[Proof of Lemma \ref{lem:mr03}.] We split the proof into three steps.

\medskip

\textit{Step 1 (Differential inequality for the residuum $ r $ \& $ \E | r |^2 $).} We start with the claim
\begin{align}\label{pmr08}
		d\mathbb{E} | r |^{2} - \mathbb{E} | r |^{2} d \tau
		&\lesssim \varepsilon |x| \mathbb{E}^{\frac{1}{2}} | r |^{2} \frac{ d \tau }{ \tilde\lambda }
		+ |x|^{2} \mathbb{E}^{\frac{1}{2}} | r |^{2} \frac{ d \tau }{ L } 
		+ \varepsilon^{2} |x|^{2} \frac{ d \tau }{ \tilde{\lambda}^{2} }
		+ |x|^{4} \frac{ d \tau }{ L^{2} } ,
	\end{align}
which we integrate in the next step.

\medskip

 By using the SDE \eqref{E: equation of mu}, and stationarity of $ \nabla d \phi $, we can write
	\begin{align*}
		[dr \cdot dr] = r^{i} r^{j} [ \partial_{i} d\phi  \cdot \partial_{j} d\phi ]
				+ 2 r^i [ \partial_i d\phi ( 0 ) \cdot df ]
				+ [ df \cdot df ].
	\end{align*}
Upon using \eqref{wr65bis} on the first term, and the definition \eqref{E: forcing} of the forcing $ df $ on the second term, this equation becomes
	\begin{align*}
		[ dr \cdot dr ]
		= | r |^{2} d \tau
		&+2  r^i \tilde{\phi}^{j}(x) [ \partial_{i} d\phi (0) \cdot (\partial_{j} d\phi (x) - \partial_{j} d\phi (0) ) ] \\
		&+ 2 r^i [ \partial_i d\phi (0) \cdot (d\phi (x) - d\phi (0) - x . \nabla d\phi (0) ) ] + [ df \cdot df ] .
	\end{align*}
Taking the estimates \eqref{E: covariation estimate 1} and \eqref{E: covariation estimate 2} for granted for now, we bound the first two error terms 
	\begin{align*}
		\mathbb{E} r^i \tilde{\phi}^j(x) [ \partial_i d\phi (0) \cdot (\partial_{j} d\phi(x) - \partial_{j} d\phi (0) ) ]
		&\lesssim |x| \big( \mathbb{E}| r |^{2} \mathbb{E}|\tilde{\phi}|^{2} \big)^{\frac{1}{2}} \frac{d \tau }{L}, \\
		\mathbb{E} r^i [ \partial_{i} d\phi (0) \cdot ( d\phi (x) - d\phi (0) - x . \nabla d\phi (0) ) ]
		&\lesssim | x |^{2}  \mathbb{E}^{\frac{1}{2}} | r |^{2} \frac{d \tau}{L},
	\end{align*}
whilst the quadratic variation $ [ df \cdot df ] $ can be estimated by first invoking the definition \eqref{E: forcing} of $ df $ and then applying \eqref{E: covariation estimate 1 b} and \eqref{E: covariation estimate 2 b}, so that
	\begin{align*}
		\mathbb{E}[ df \cdot df ] &\lesssim \mathbb{E}\tilde{\phi}^{i}(x) \tilde{\phi}^{j}(x) [ (\partial_i d\phi (x) - \partial_{i} d\phi (0) ) \cdot (\partial_{j} d\phi (x) - \partial_{j} d\phi (0) ) ]  \phantom {  \frac{d \tilde{\lambda}^{2}}{L^2 \tilde{\lambda}^{2}} }
 \\
			&+ [(d\phi (x) - d\phi (0) - x.\nabla d\phi (0) ) \cdot (d\phi (x) - d\phi (0) - x.\nabla d\phi (0) )] \\
			& \lesssim |x|^{2} \mathbb{E} |\tilde{\phi}|^{2} \frac{d \tau}{L^{2}}
			+ | x |^{4} \frac{d \tau}{L^2} .
	\end{align*}
It remains to notice that $ d \E | r |^2 = \E [ dr \cdot dr ] $ so that by putting it all together, we find
	\begin{align*}
		d\mathbb{E} | r |^{2}  - \mathbb{E} | r |^{2} d \tau
		&\lesssim |x| \big( \mathbb{E}| r |^{2} \mathbb{E}|\tilde{\phi}|^{2} \big)^{\frac{1}{2}} \frac{d \tau}{L}
		+  |x|^{2} \mathbb{E}^{\frac{1}{2}} | r |^{2} \frac{d \tau}{L}
		+ |x|^{2} \E | \tilde\phi |^2 \frac{ d \tau }{ L^2 }
		+ |x|^{4} \frac{ d \tau }{L^{2}}.
	\end{align*}
Appealing again to the bound on $\tilde{\phi}$ in Lemma \ref{proxy_est}, we can rearrange this to obtain \eqref{pmr08}.

\medskip

\textit{Step 2 (Estimate on the initial data $ \E | r_* |^2 $).}  Next, we estimate the initial condition by
	\begin{align} \label{bound_ic}
		\frac{1}{|x|^{2}} \mathbb{E} | r_* |^{2} \lesssim \frac{\varepsilon^{2}}{\tilde{\lambda}_{*}^{2}}.
	\end{align}
	for $ r_* = r_{L_{*}}$. Toward that end, we invoke the triangle inequality, stationarity of $\tilde{\phi}_{L_{*}}$, and our estimate of $\mathbb{E} |\tilde{\phi}_{ L_* } |^{2}$, see Lemma \ref{proxy_est}, to find 
	\begin{align*}
		\mathbb{E} | r_{L_{*}} |^{2} &= \mathbb{E} |\tilde{\phi}_{L_{*}}(x) - \tilde{\phi}_{L_{*}}(0)|^{2} \leq 4 \mathbb{E}|\tilde{\phi}_{L_{*}}|^{2} \lesssim \frac{\varepsilon^{2} L_{*}^{2}}{\tilde{\lambda}_{*}^{2}}.
	\end{align*}
Since $L_{*} = |x|$ by definition, this yields the estimate \eqref{bound_ic}.

\medskip
	
\textit{Step 3 (Estimate of $ \mathbb{E} | r |^{2} $).} Next, we use an ODE argument to infer
\begin{align}\label{pmr09}
	\frac{1}{ | x |^2 } \E | r_{L} |^2 
	\lesssim \frac{ \tilde\lambda }{ \tilde\lambda_* } \bigg( \frac{1}{ | x |^2 } \E | r_* |^2
	+ \frac{ \varepsilon^2 } { \tilde\lambda_*^2 }  \Big( 1
	+ \frac{ \varepsilon^2  }{ \tilde\lambda_*^{ 2 } } \Big) \bigg) .
\end{align}
To this end, we switch to the variables $ R $ and $ \tau $, where
\begin{align*}
	R = \E | r |^2 .
\end{align*}
Upon multiplying the differential inequality \eqref{pmr08} by the integrating factor $ \frac{1}{\tilde\lambda} = \exp( - \tau ) $ and recalling that $ \frac{ d }{ d \tau } \frac{ R }{ \tilde\lambda } = \frac{ 1 }{ \tilde\lambda } ( \frac{ d }{ d \tau } R - R ) $, the equation becomes
\begin{align*}
\frac{ d }{ d \tau } \frac{ R }{ \tilde\lambda }
\lesssim \varepsilon \frac{ | x | R^{ \frac{1}{2} } }{ \tilde\lambda^2 }
+ \frac{ | x |^2 R^{ \frac{1}{2} } }{ L \tilde\lambda }
+ \varepsilon^2 \frac{ | x |^2 }{ \tilde\lambda^3 }
+ \frac{ | x |^4 }{ L^2 \tilde\lambda } ,
\end{align*}
which implies after integration that
\begin{align*}
	\frac{ R }{ \tilde\lambda }  - \frac{ R_* }{ \tilde\lambda_* }
	& \lesssim \Big( \varepsilon | x |  \int_{ \tau_* }^{ \tau } d\tau \frac{1}{ \tilde\lambda^{ \frac{3}{2 } } }
	+ \frac{ | x |^2 }{ \tilde\lambda_*^{ \frac{5}{2} } }  \int_{ \tau_* }^{ \tau } d\tau \frac{ \tilde\lambda^2 }{ L } 
	\Big) \sup_{ [ \tau_*, \tau] } \Big( \frac{ R }{ \tilde\lambda  } \Big)^{ \frac{1}{2} }
	+ \varepsilon^2 | x |^2  \int_{ \tau_* }^{ \tau } d\tau \frac{1}{ \tilde\lambda^{ 3 } }
	+ \frac{ | x |^4 }{ \tilde\lambda_*^3 } \int_{ \tau_* }^{ \tau } d\tau \frac{ \tilde\lambda^2 }{ L^2 } .
\end{align*}
Since $\frac{d \tau}{\tilde{\lambda}^{p}} = \frac{1}{2} \frac{d \tilde{\lambda}^{2}}{(\tilde{\lambda}^{2})^{1 + \frac{p}{2}}}$ and $d \tau \frac{\tilde{\lambda}^{2}}{L^{p}} = \frac{1}{2} d \tilde{\lambda}^{2} \exp ( - \frac{p}{\varepsilon^{2}} (\tilde{\lambda}^{2} - 1) )$, we have
\begin{align*}
\int_{ \tau_* }^{ \tau } d\tau\, \frac{ \tilde\lambda^2 }{L^p} \lesssim \frac{ \varepsilon^2 }{ p L_*^p }
\quad { \rm and } \quad
\int_{ \tau_* }^{ \tau } d\tau \, \frac{1}{ \tilde\lambda^p } \lesssim \frac{ 1 }{ p \tilde\lambda_*^{p} }
\quad { \rm for ~ any } ~ p > 0 ,
\end{align*}
so that we may post-process the above estimate to
\begin{align*}
	\frac{ R }{ \tilde\lambda }  - \frac{ R_* }{ \tilde\lambda_* }
	& \lesssim \varepsilon \Big( \frac{ | x | }{ \tilde\lambda_*^{ \frac{3}{2 } } }
	+ \varepsilon \frac{ | x |^2 }{ L_* \tilde\lambda_*^{ \frac{5}{2} } }
	\Big) \sup_{ [ \tau_*, \tau] } \Big( \frac{ R }{ \tilde\lambda  } \Big)^{ \frac{1}{2} }
	+ \varepsilon^2 \Big( \frac{  | x |^2 }{ \tilde\lambda_*^{ 3 } }
	+  \frac{ | x |^4 }{ L_*^2 \tilde\lambda_*^3 } \Big) .
\end{align*}
Using Young's inequality and the fact that $ | x | = L_* $, the above inequality implies \eqref{pmr09}.
\end{proof}

\subsection{Appendix: Estimates on Quadratic Variations}\label{sec:appendix-quad-var}

Here we prove the covariation estimates \eqref{E: covariation estimate 1},  \eqref{E: covariation estimate 1 b}, \eqref{E: covariation estimate 2} and \eqref{E: covariation estimate 2 b} that were needed above. To this end, let us recall from \cite[Appendix A]{OW24} that
\begin{align}\label{app01}
	[ { \rm tr } ( \nabla d \phi )^* \nabla d \phi ] \lesssim d \tau 
	\quad { \rm and } \quad 
	\sum_{ i = 1 }^{ 2 } \, [ { \rm tr } ( \nabla \partial_i d \phi )^* \nabla \partial_i d \phi ] \lesssim \frac{ d \tau }{ L^2 } .
\end{align}
In this text, the first item may, in view of definition \eqref{ao37}, be inferred from \eqref{ao51}, whilst the second item follows from the first as in \cite[argument for (65)]{OW24}. The logic is that every derivative gains a factor of $ \frac{1}{L} $. We will use these estimate, to first derive \eqref{E: covariation estimate 1 b}. It is then a routine computation that we include for the readers convenience to derive \eqref{E: covariation estimate 1}, \eqref{E: covariation estimate 2} and \eqref{E: covariation estimate 2 b}.

\medskip

\textit{Argument for \eqref{E: covariation estimate 1 b}.} Upon applying the Kunita-Watanabe-/ Cauchy-Schwarz-inequality, we may restrict to the case $ i = j $. To compute the quadratic variation, we first exploit the independence of $ \nabla d \phi $'s increments, and appeal to the mean-value inequality afterwards. By stationarity and \eqref{app01} this implies
	\begin{align*}
		& \Big[ \int_{ \tau_- }^{ \tau_+  } (\partial_{i} d \phi (x) - \partial_{i} d \phi (0) ) \cdot \int_{ \tau_- }^{ \tau_+  } (\partial_{i} d \phi (x) - \partial_{i} d \phi ( 0 ) ) \Big]
		= \mathbb{E} \Big| \int_{ \tau_- }^{ \tau_+  } (\partial_{i} d \phi (x) - \partial_{i} d \phi (0) ) \Big|^2 \\
		& = \mathbb{E}  \Big| x . \int_{0}^{1} ds  \int_{ \tau_- }^{ \tau_+  } \nabla \partial_{i} d \phi (s x) \Big|^{2} 
		\leq |x|^{2} \mathbb{E} \Big|  \int_{ \tau_- }^{ \tau_+  } \nabla \partial_{i} d \phi  \Big|^{2} 
		\lesssim |x|^{2}  \int_{ \tau_- }^{ \tau_+  }\frac{ d \tau }{ L^2 }  .
	\end{align*}
In view of the arbitrariness of the scales $ \tau_{-} $ and $ \tau_+ $ this implies \eqref{E: covariation estimate 1 b}.

\medskip

\textit{Argument for \eqref{E: covariation estimate 1}.} Recall that, by the Kunita-Watanabe inequality
	\begin{align*}
		& \Big|  \int_{ \tau_- }^{ \tau_+  } [ \partial_i d \phi (0) \cdot (\partial_j d \phi (x) - \partial_j d \phi ( 0 ) )] \Big| \\
		&\leq \Big( \int_{ \tau_- }^{ \tau_+  } \frac{1}{L} [ \partial_i d \phi \cdot \partial_i d \phi ] \Big)^{\frac{1}{2}}
		\Big(  \int_{ \tau_- }^{ \tau_+  } L [ (\partial_{j} d \phi (x) - \partial_{j} d \phi (0) ) \cdot ( \partial_j d \phi (x) - \partial_{j} d \phi (0 )) ] \Big)^{\frac{1}{2}}
	\end{align*}
	for any two scales $ \tau_{-}, \tau_{+} $. Hence, we may deduce from \eqref{E: covariation estimate 1 b} and the first item in \eqref{app01} that 
	\begin{align*}
		\Big|  \int_{ \tau_- }^{ \tau_+  }  [ \partial_{j} d \phi^{i} (0) \, (\partial_{\ell} d \phi^{i}(x) - \partial_{\ell} d \phi^{i}(0) )] \Big| \lesssim | x |  \int_{ \tau_- }^{ \tau_+  } \frac{ d\tau }{ L }.
	\end{align*}
Again, since the scales $ \tau_{-} $ and $ \tau_{+} $ are arbitrary this implies \eqref{E: covariation estimate 1}.

\medskip

\textit{Argument for \eqref{E: covariation estimate 2}.} Again, we proceed by Taylor, this time in form of
\begin{align*}
 \int_{ \tau_- }^{ \tau_+  } ( d \phi  (x) - d \phi (0) - x.\nabla \phi (0) ) 
=  \int_0^1 ds  \int_{ \tau_- }^{ \tau_+  } x . \nabla ( d \phi ( s x ) - d \phi (0) ) .
\end{align*}
By virtue of Jensen's inequality, followed by \eqref{E: covariation estimate 1 b}, this identity implies
\begin{align*}
	& \Big[  \int_{ \tau_- }^{ \tau_+  } ( d \phi (x) - d \phi (0) - x . \nabla d \phi (0) ) \cdot  \int_{ \tau_- }^{ \tau_+  } ( d \phi (x) - d \phi (0) - x . \nabla d \phi (0) )  \Big] \\
	& \leq \int_0^1 ds \Big[  \int_{ \tau_- }^{ \tau_+  } x . \nabla ( d \phi ( s x ) - d \phi (0) ) \cdot   \int_{ \tau_- }^{ \tau_+  } x . \nabla ( d \phi ( s x ) - d \phi (0) ) \Big]
	\lesssim |x|^{4}  \int_{ \tau_- }^{ \tau_+  } \frac{ d \tau }{ L^{2} },
	\end{align*}
as claimed.

\medskip

\textit{Argument for \eqref{E: covariation estimate 2}.} Similar to the argument for \eqref{E: covariation estimate 1}, Kunita-Watanabe inequality, \eqref{E: covariation estimate 2 b} and the second item in \eqref{app01} imply \eqref{E: covariation estimate 2}.

\section*{Acknowledgements}

We thank B\'{a}lint T\'{o}th, Fabio Toninelli, and Giuseppe Cannizzaro for helpful discussions.  We also benefitted from conversations with Tomasz Komorowski and Ofer Zeitouni.

\end{document}